\documentclass[12pt]{amsart}

\usepackage{amsmath}
\usepackage{amsthm}
\usepackage{amsfonts}
\usepackage{amssymb,amsmath,latexsym}
\usepackage{graphicx}
\usepackage{cite}

\theoremstyle{plain}
\newtheorem{theorem}{Theorem}%[section]
\newtheorem{lemma}[theorem]{Lemma}%[section]
\newtheorem{corollary}[theorem]{Corollary}%[section]

\theoremstyle{definition}
\title[Khintchine-Pollaczek formula]{Khintchine-Pollaczek formula for random walks whose steps have one geometric tail}

\author{Robert~O. Bauer}
\address{Altgeld Hall\\Department of Mathematics\\ 	University of Illinois at Urbana-Champaign\\ 	1409 West Green Street \\ Urbana, IL 61801, USA}\email{rbauer13@illinois.edu}



\thanks{Research supported in part by NSA grant H98230-10-1-0193}

\begin{document}

\maketitle

\begin{abstract}
We derive a Khinchine-Pollaczek formula for random walks whose steps have a geometric left tail. The construction rests on the memory-less property of the geometric distribution. An example from a tandem queue modeling dynamic instability for microtubules is given. 
\end{abstract}

\section{Introduction}

Let $X_1,X_2,\dots,$ be a sequence of independent random variables with common distribution $F$ not concentrated on a half-axis. The {\em induced random walk} is the sequence of random variables 
\[
S_0=0,\quad S_n=X_1+\cdots+X_n.
\]
If $\mathbb E[X_n]<0$, then $\sup_n S_n<\infty$, $P-a.s.$ An explicit expression for the function
\[
s\mapsto\mathbb E\left[s^{\sup_n S_n}\right]
\]
is a Khintchine-Pollaczek formula for the random walk $\{S_n\}$.

A motivation to calculate a Khintchine-Pollaczek formula arises in queueing theory. Consider a single server queue, where the times between the arrival of individual customers are assumed to be independent and identically distributed random variables $\sigma_1,\sigma_2,\dots$. The time it takes the server to serve the $n$-th customer is $\tau_n$, and we assume that $\tau_1,\tau_2,\dots$ are independent and identically distributed random variables which are independent of the inter-arrival times $\sigma_n$. If the server is empty when a customer arrives, then the customer is serviced immediately. Otherwise, the customer joins the queue in front of the server. If $w_n$ denotes the waiting time of the $n$-th customer, from the time of her arrival until the time her service commences, then it follows readily that
\[
w_{n+1}=\max(0, w_n+\tau_n-\sigma_n),
\] 
where $\sigma_n$ denotes the inter-arrival time between the $n$-th and $n+1$-st customer. Setting  $X_n=\tau_n-\sigma_n$, consider the induced random walk $\{S_n\}$. Then 
\begin{align}
\max_{k\le n+1}S_k&=\max(0,X_1,X_1+X_2,\dots,X_1+\cdots+X_{n+1})\notag\\
&\overset{law}=\max(0,X_{n+1},X_{n+1}+X_n,\dots,X_{n+1}+\cdots+X_1)\notag\\
&=\max(0,X_{n+1}+\max(0,X_n,X_n+X_{n-1},\dots,X_n+\cdots+X_1))\notag\\
&\overset{law}=\max(0,X_{n+1}+\max_{k\le n}S_k),
\end{align}
where $\overset{law}=$ denotes equality in law of the random variables to the right and left of the equality sign. Whence 
\[
w_n\overset{law}=\max_{k\le n}S_k,
\]
and 
\[
P(\sup_n S_n>x)=\uparrow\lim_n P(\max_{k\le n}S_k>x)=\uparrow \lim_n P(w_n>x).
\]
Thus knowledge of the law of $\sup_n S_n$ gives the limiting probabilities for the waiting time sequence $w_n$.

There are a few classic cases where a Khintchine-Pollaczek formula can be computed, see \cite{Asmussen:2003} and \cite[XII.5]{Feller:1971}. For example, if the queue is of the M/M/1-type, i.e. both inter-arrival times and service times are exponential, then the Khintchine-Pollaczek formula is known. It is also known if only the $\sigma$s or only the $\tau$s are exponential while the other is continuous (``one-sided exponential tail"). If the walk can only take integer steps, then the Khintchine-Pollaczek formula is known in principle---using a partial fraction decomposition---if the step size is bounded. In this paper we derive a Khintchine-Pollaczek formula in a case where the step size is integer but the steps can be arbitrarily large. The particular steps we will be considering have a geometric tail on one side. That such a formula should exist in this case is not surprising. Indeed, the key to the proof in the M/M/1
 case (and its one-sided extensions) is the memoryless property of the exponential distribution. The geometric distribution shares the memoryless property with the exponential distribution (indeed it is the ``discrete version'' of the exponential distribution).

\subsection{Motivation}

Queues to which our result applies can arise as follows. Consider two coupled single server queues: Queue 1 is an M/M/1-queue, with inter-arrival times $\sigma_n$ exponential with parameter $1/\alpha$ and service times $\tau_n$ exponential with parameter $1/\beta$. Once a customer finishes service at the first server, he immediately joins Queue 2. The server for Queue 2 becomes active ONLY if Server 1 is idle, i.e. if there are no customers in Queue 1. When Server 2 is active, the service times $\tilde{\tau}_n$ are exponential with parameter $1/\gamma$. We assume that $\beta+\gamma<\alpha$, i.e. that the average time that it takes for a customer to be serviced by both, Server 1 and Server 2, is less than the average inter-arrival time between customers  arriving at Server 1.

Then the number of customers $N$ in Queue 2 when Server 2 commences service equals the number of customers serviced by Server 1 during its initial ``busy period.'' Since Server 1 is M/M/1, the distribution of $N$ is known explicitly, see \eqref{E:distn}. Since Server 2 is only active as long as no new customers arrive at Server 1, Server 2 will only complete service of a customer if $\sigma_N>\tilde{\tau}_1$. Since both these variables are exponential and independent, this probability is $\alpha/(\alpha+\gamma)$. Furthermore, by the memoryless property of the exponential distribution, the probability that Server 2 completes service of $m$ customers in its queue before a new customer arrives at Server 1 but not $m+1$ is 
\begin{equation}\label{E:geo}
\left(\frac{\alpha}{\alpha+\gamma}\right)^m\frac{\gamma}{\alpha+\gamma}.
\end{equation}
These are the probabilities for the geometric distribution. Consider now the number of customers $Z_1$ in Queue 2 at the time when it becomes idle, i.e. when a customer arrives at Server 1.  Then 
\[
Z_1=\max(N_1-M_1,0),
\] 
where $N_1$ is the number of customers served by Server 1 during the initial busy period, and $M_1$ is an independent geometric random variable with distribution \eqref{E:geo}. Similarly, if $Z_k$ denotes the number of customers in Queue 2 after its $k$-th busy period, then 
\[
Z_{k+1}=\max(0, Z_k+N_{k+1}-M_{k+1}),
\]
where $N_1,N_2,\dots$ and $M_1,M_2,\dots$ are independent sequences of independent and identically distributed random variables.  Thus the sequence $Z_{k}$ forms a ``waiting time sequence" of an abstract queue with inter-arrival times $N_k$ and service times $M_k$. The associated random walk takes integer steps with a distribution with a geometric left tail!   

\subsection{Further motivation} \label{SS:furthermot}

For properly chosen parameters $\alpha, \beta, \gamma$, the above tandem queue exhibits a behavior which is known as ``dynamic instability'' in microbiology, see \cite[Ch. 16]{Alberts:2008}. To observe this phenomenon, consider each customer as an {\em oriented edge}, with the orientation indicated by an arrowhead. Initially, all edges are {\em red}. Queue 1 in front of Server 1 then corresponds to a chain of red edges. We assume that edges always attach to the chain (join the queue) pointing toward the server. The service provided by Server 1 consists of changing the color of an edge from {\em red} to {\em blue}. Once an edge has changed color we consider it part of Queue 2 in front of Server 2. Thus, while Server 1 is busy, we obtain a chain of oriented edges, all pointing in the same direction---say towards the {\em plus-end}---consisting of a blue part (Queue 2) and a red part (Queue 1). For descriptive purposes, we call this red part of the chain the ``cap.'' The cap disappears precisely when Server 1 becomes idle. Then Server 2 becomes active. The service provided by Server 2 is {\em dissociation}, and the service is provided in the order ``last come, first served.'' Blue edges dissociate from the {\em minus-end} of the chain until a red edge arrives at Server 1---at this instant Server 1 corresponds to the minus-end---thereby {\em capping} the chain. In analogy with the situation for microtubules in cells, we call the period when Server 2 is busy and the chain shortens {\em catastrophe}, and the arrival of a red edge capping the shortening chain and initiating a busy period for Server 1 {\em rescue}. Dynamic instability is the ``random'' alternating of catastrophe and rescue.

The term dynamic instability was coined for protein polymers such as microtubules formed by the linear association of identical protein subunits (tubulin monomers). Dynamic instability refers to the random alternating of a longer and slower growth phase (polymerization) with a shorter and faster shrinking phase (depolymerization). The monomers are either bound to GTP or GDP (Guanosine tri/bi-phosphate). GTP-bound monomers tend to polymerize while GDP-bound monomers tend to depolymerize.  Furthermore, the monomers have an orientation, with a plus end and a minus end, and they attach so that a plus end joins with a minus end. Once a GTP-bound monomer attaches to a polymer, the polymer exerts a GTP-ase action on the monomer and hydrolyzes the GTP to GDP. 

In the queueing theory model, the GTP-bound monomers are red oriented edges, the GDP-bound monomers are blue and the GTP-ase action of the polymer is the service provided by Server 1. The queueing discipline of Server 2, ``last come, first served,'' corresponds for microtubules to what is called a {\em vectorial} model for hydrolysis.

%Many results in queueing theory are based on the connection of the waiting times process and an associated random walk. If the queue is an M/M/1 queue, then a classical result for the associated random walk is the Khintchine-Pollaczek formula, \cite[XII.5]{Feller:1971}. The derivation exploits the ``memoryless'' property of the exponential distribution. In this paper we obtain a discrete version of this result, using the memoryless property of the geometric distribution. The discreteness of the distribution introduces a number of additional subtleties, but in general our proof is organized as the proof of the continuous case in \cite{Billingsley:1986}. 

%The motivation to consider the discrete case came from our attempt to model what is known as ``dynamic instability'' of microtubules (see \cite[Ch. 16]{Alberts:2008}) by a tandem queue. While the primary queue of the tandem is M/M/1, the secondary queue is not. However, the primary queue bequeathes a geometric left tail to the secondary queue and this caused us to search for a Khintchine-Pollaczek formula in that setting.   

\subsection{Structure of this paper} In section \ref{S:pre}, we introduce the notation and list some results from the literature we will use. The mathematical core of this paper is section \ref{S:geotails}, where we obtain the Khintchine-Pollaczek formula. The proof is structured as for the analogous result in the M/M/1 case. A difficulty not arising in the M/M/1 case (or the one-sided extensions, M/G/1, G/M/1) is that we have to deal with ties, which require making the distinction between strict and weak ladder indices and heights. This circumstance complicates the argument as well as the formula we obtain. Finally, in section \ref{S:example}, we apply our results to the example of the tandem queue modeling dynamic instability described above. 
 
\section{Preliminaries}\label{S:pre}

Let $X_1,X_2,\dots,$ be a sequence of independent random variables with common distribution $F$ not concentrated on a half-axis, and denote $\{S_n\}$ the induced random walk.
The integer $n$ is a (strict) ladder index for the random walk if 
\[
\max_{0\le k< n}S_k<S_n.
\] 
If $n$ is a ladder index, then $S_n$ is the {\em ladder height} associated with $n$. For Borel sets $A\subset(0,\infty)$, define a finite measure $L$ by
\begin{equation}
L(A)=\sum_{n=1}^{\infty} P\left(\max_{0\le k< n}S_k=0<S_n\in A\right).
\end{equation}
The probability that there is at least one ladder index is 
\begin{equation}\label{E:defp}
p=L(0,\infty)=P\left(\sup_n S_n>0\right).
\end{equation}
Let $T_1$ be the first ladder index, and denote $H_1$ the first ladder height, i.e. $H_1=S_{T_1}$. These variables are defective with probability $1-p$  and remain undefined if there is no first ladder index. In fact, for $x>0$,
\begin{equation}\label{E:LH1}
L[x,\infty)=P(H_1\ge x).
\end{equation}

Following the notation in \cite[Chapter XII]{Feller:1971}, we call the smallest $n$ such that $S_1<0,\dots,S_{n-1}<0$, but $S_n\ge0$ the {\em first weak ladder index} and denote it by $\overline{T}_1$. The corresponding {\em weak ladder height} is denoted by $\overline{H}_1$, so that $\overline{H}_1=S_{\overline{T}_1}$. Again, these variables are possibly defective. We set 
\begin{equation}\label{E:defzeta}
\zeta=P\left(\overline{H}_1=0\right)=\sum_{n=1}^\infty P\left(\max_{1\le k<n}S_k<0, S_n=0\right).
\end{equation}
Since $X_1>0$ implies $\overline{H}_1>0$ and it is assumed that $X_1$ is not concentrated on a half-axis, it follows that $0\le\zeta<1$.

\begin{lemma}
For $x>0$, we have
\begin{equation}\label{E:factorzeta}
P\left(\overline{T}_1<T_1,H_1\ge x\right)=\zeta P\left(H_1\ge x\right).
\end{equation} 
\end{lemma}

\begin{proof}
If $0<k<n$ and $P\left(\overline{T}_1=k,\overline{H}_1=0\right)>0$, then 
\[
P\left(T_1=n, H_1\ge x|\overline{T}_1=k,\overline{H}_1=0\right)=P\left(T_1=n-k,H_1\ge x\right)
\]
by the Markov property. Thus
\begin{align}
P&\left(\overline{T}_1<T_1,H_1\ge x\right)\notag\\
&=\sum_{n=2}^\infty\sum_{k=1}^{n-1}P\left(\overline{T}_1=k,T_1=n,H_1\ge x\right)\notag\\
&=\sum_{k=1}^\infty\sum_{n=k+1}^\infty P\left(\overline{T}_1=k,\overline{H}_1=0\right)P\left(T_1=n,H_1\ge x|\overline{T}_1=k,\overline{H}_1=0\right)\notag\\
&=\sum_{k=1}^\infty P\left(\overline{T}_1=k,\overline{H}_1=0\right)\sum_{n=1}^\infty P\left(T_1=n,H_1\ge x\right).
\end{align}

\end{proof}

Denote ${L^n}^*$ the $n$-fold convolution of $L$ with itself and ${L^0}^*$ a unit mass at the point $0$. Define the measure $\psi$ by
\begin{equation}\label{E:defpsi}
\psi(A)=\sum_{n=0}^\infty {L^n}^*(A),\quad A\subset[0,\infty).
\end{equation}
We have, cf. \cite[Theorem 24.2(iii)]{Billingsley:1986},

\begin{theorem}
If $p<1$, then with probability $p^n(1-p)$ there are exactly $n$ ladder indices; with probability 1 there are only finitely many ladder indices and $\sup_n S_n<\infty$; finally
\begin{equation}\label{E:lawofsup}
P\left(\sup_{n\ge0}S_n\in A\right)=(1-p)\psi(A),\quad A\subset[0,\infty).
\end{equation}
\end{theorem}

Furthermore, cf. \cite[Theorem 24.3]{Billingsley:1986},

\begin{theorem}
The measure $\psi$ satisfies
\begin{equation}\label{E:recursion}
\int_{y\le x}\psi[0,x-y]\ dF(y) = \sum_{n=1}^\infty P\left(\min_{1\le k<n} S_k > 0, S_n\le x\right).
\end{equation}
\end{theorem}

\section{One-sided geometric tails} \label{S:geotails}

We assume from now on that the random variables $X_i$ are integer valued.

\begin{theorem}[Geometric right tail]
Suppose that $\mathbb E[X_1]<0$ and that the right tail of $F$ is geometric:
\begin{equation}\label{E:georight}
P(X_1\ge x)=\xi\ r^x,\quad x=0,1,\dots,
\end{equation} 
where $0<\xi<1$ and $0<r<1$. Then $p<1$ and 
\begin{equation}\label{E:geolaw}
P\left(\sup_{n\ge0}S_n>x\right)=p\left[1-(1-p)(1-r)\right]^x,\quad x=0,1,\dots.
\end{equation}
Moreover, $1/[1-(1-p)(1-r)]$ is the unique root of the equation
\begin{equation}\label{E:identifyp}
\sum_{x=-\infty}^\infty s^x P(X_1=x)=1
\end{equation}
in the range $1<s<1/r$.
\end{theorem} 

\begin{proof}
For $x\ge0$ we have
\begin{align}
&P\left(\max_{0\le k<n}S_k\le0,S_n>x\right)\notag\\
&=P\left(S_n>x|S_n>0,\max_{0\le k<n}S_k\le0\right)P\left(\max_{0\le k<n}S_k\le0,S_n>0\right),
\end{align}
and also
\begin{align}
&P\left(S_n>x|S_n>0,\max_{0\le k<n}S_k\le0\right)\notag\\
&=P\left(X_n>x-S_{n-1}|X_n>-S_{n-1},\max_{0\le k<n}S_k\le0\right).
\end{align} 
Furthermore, since $X_n$ is independent of $S_0,\dots,S_{n-1}$, and the geometric distribution is {\em memoryless},
\begin{align}
&P\left(X_n>x-S_{n-1}|X_n>-S_{n-1},\max_{0\le k<n}S_k\le0\right)\notag\\
&=\frac{1}{\xi}P(X_n\ge x)=r^x.
\end{align}
Whence
\begin{equation}
P\left(\max_{0\le k<n}S_k\le0,S_n>x\right)=r^xP\left(\max_{0\le k<n}S_k\le0<S_n\right),
\end{equation}
and so 
\begin{align}
L((x,\infty))&=\sum_{n=1}^\infty P\left(\max_{0\le k<n}S_k\le0,S_n>x\right)\notag\\
&=r^x\sum_{n=1}^\infty P\left(\max_{0\le k<n}S_k\le0<S_n\right)=r^x p.
\end{align}
This shows that the measure $L$ is $p$ times the geometric distribution with parameter $1-r$. Hence $L^{n*}$ is $p^n$ times the negative binomial distribution with parameters $n$ and $1-r$: For $x>0$,
\[
L^{n*}\{x\}=
\begin{cases}
p^n\binom{x-1}{n-1}(1-r)^n r^{x-n}, &\text{if } 1\le n\le x;\\
0, &\text{otherwise}.
\end{cases}
\] 
In particular, for $x=1,2,\dots$,
\[
\psi\{x\}=\sum_{n=1}^x p^n\binom{x-1}{n-1}(1-r)^n r^{x-n}=p(1-r)[1-(1-p)(1-r)]^{x-1},
\]
so that 
\begin{align}\label{E:psi}
\psi([0,x])&=1+\sum_{k=1}^x p(1-r)[1-(1-p)(1-r)]^{k-1}\notag\\
&=
\begin{cases}
1+(1-r)x, &\text{if $p=1$;}\\
\frac{1}{1-p}-\frac{p}{1-p}\left[1-(1-p)(1-r)\right]^x, &\text{if $p<1$.}
\end{cases}
\end{align}
Since $\mathbb E[X_1]<0$, we have $S_n\to-\infty$ $P$-a.s and $p<1$. The identity \eqref{E:geolaw} now follows from \eqref{E:lawofsup}.

Next, since $S_n\to-\infty$, the right side of \eqref{E:recursion} is equal to one for $x=0$, so that by \eqref{E:psi}
\begin{equation}\label{E:step}
\sum_{y=-\infty}^0\left(1-p\left[1-(1-p)(1-r)\right]^{-y}\right)F\{y\}=1-p.
\end{equation}
Denote $f(s)$ the probability generating function of $X_1$ given on the left in \eqref{E:identifyp}. Because of \eqref{E:georight}, it is defined for $0\le s<1/r$ and 
\[
f(s)=\sum_{y=-\infty}^{-1}s^y F\{y\}+\frac{\xi(1-r)}{1-rs}.
\]
Since $1<1/[1-(1-p)(1-r)]<1/r$, it follows
from this, \eqref{E:step}, and $F(-1)=1-\xi$, that $f(1/[1-(1-p)(1-r)])=1$.

Finally, $f(1)=1$ and, since $P(X_1<0)>0$, 
\[
f''(s)=\sum_{\mathbb Z\backslash\{0,1\}}x(x-1)s^{x-2}F\{x\}>0
\]
for $0< s<1/r$. Hence equation \eqref{E:identifyp} cannot have more than one root in $(1,1/r)$.

\end{proof}

\begin{theorem}[Geometric left tail] \label{T:geoleft}
Suppose that $\mathbb E[X_1]<0$ and that the left tail of $F$ is geometric:
\begin{equation}\label{E:geoleft}
P(X_1\le x)=\xi r^{-x},\quad x=0,-1,\dots,
\end{equation}
where $0<\xi<1$ and $0<r<1$. Then 
\begin{equation}\label{E:findp}
(1-\zeta)p=r+(1-r)\mathbb E[X_1],
\end{equation}
and for $x=1,2,\dots$,
\begin{equation}\label{E:Lx}
(1-\zeta)L(0,x]=F(x)-F(0)+(1-r)\sum_{m=1}^x(1-F(m)).
\end{equation}
\end{theorem}

\begin{proof}
Assume first not \eqref{E:geoleft} but \eqref{E:georight}, and assume also that $\mathbb E[X_1]>0$. In that case, $S_n\to\infty$, $P$-a.s., so that $p=1$ by \eqref{E:defp}. Thus, from \eqref{E:psi}, we have  $\psi[0,x]=1+(1-r)x$ for $x\ge0$. Whence, for $x\ge0$ (cf. \eqref{E:recursion}),
\begin{align}\label{E:recleft}
\sum_{n=1}^\infty& P\left(\min_{1\le k<n} S_k>0,S_n\le x\right)\notag\\
&=\sum_{m\le x}\left[1+(1-r)(x-m)\right]P(X_1=m)\notag\\
&=[1+(1-r)x]P(X_1\le x)-(1-r)\sum_{m\le x}m P(X_1=m).
\end{align}

If \eqref{E:geoleft} holds with $\mathbb E[X_1]<0$, then the above applies to the sequence $\{-X_n\}$. In particular, if $X_n$ is replaced by $-X_n$, $x$ by $-x$, then \eqref{E:recleft} becomes
\begin{align}\label{E:recleft-}
\sum_{n=1}^\infty& P\left(\max_{1\le k<n} S_k<0,S_n\ge x\right)\notag\\
&=[1-(1-r)x]P(X_1\ge x)+(1-r)\sum_{m\ge x}m P(X_1=m).
\end{align}

Recall the definitions of the weak/strict first ladder index. We have for $x>0$
\begin{align}
P&\left(\max_{0\le k<n}S_k=0<S_n,S_n\ge x\right)-P\left(\max_{1\le k<n}S_k<0,S_n\ge x\right)\notag\\
&=P\left(\overline{T}_1<T_1=n,S_n\ge x\right).
\end{align}
Thus, from \eqref{E:factorzeta} and \eqref{E:LH1},
\begin{align}\label{E:Lzeta}
L[x,\infty)&=\sum_{n=1}^\infty P\left(\max_{1\le k<n}S_k<0,S_n\ge x\right)+\sum_{n=1}^\infty P\left(\overline{T}_1<T_1=n, S_n\ge x\right)\notag\\
&=\sum_{n=1}^\infty P\left(\max_{1\le k<n}S_k<0,S_n\ge x\right)+\zeta L[x,\infty).
\end{align}
Since
\[
\sum_{m\ge x} m P(X_1=m)=x P(X_1\ge x)+\sum_{m\ge x}(1-F(m)),
\] 
it follows from \eqref{E:recleft-} and \eqref{E:Lzeta} that 
\begin{equation}\label{E:Lzx}
(1-\zeta)L[x,\infty)=1-F(x-1)+(1-r)\sum_{m\ge x}(1-F(m)).
\end{equation}
From \eqref{E:geoleft},
\begin{equation}
\sum_{m\le0}m P(X_1=m)=-\xi r(1-r)\sum_{m\ge1}m r^{m-1}=-\xi\frac{r}{1-r}.
\end{equation} 
Using $F(0)=\xi$, this gives
\begin{equation}
\sum_{m\ge0}(1-F(m))=\mathbb E[X_1]-\sum_{m\le0} m P(X_1=m)=\mathbb E[X_1]+\frac{r}{1-r}F(0).
\end{equation}
Hence
\begin{align}\label{E:Lz1}
(1-\zeta)L[1,\infty)&=1-F(0)+(1-r)\sum_{m\ge1}(1-F(m))\notag\\
&=r+(1-r)\mathbb E[X_1],
\end{align}
from which \eqref{E:findp} follows. Applying \eqref{E:Lz1} and \eqref{E:Lzx} to
\[
L(0,x]=L[1,\infty)-L[x+1,\infty)
\]
gives \eqref{E:Lx}.
\end{proof}

\begin{corollary}[Khinchine-Pollaczek formula]
Under the assumptions of Theorem \ref{T:geoleft}, if 
\[
\mathcal M(s)=\mathbb E\left[s^{\sup_{n\ge0}S_n}\right],\qquad \mathcal F^+(s)=\mathbb E[s^{X_1},X_1>0],
\]
then
\begin{equation}
\mathcal M(s)=\frac{1-\zeta-r-(1-r)\mathbb E[X_1]}{1-\zeta-\left[\left(1-\frac{1-r}{1-s}\right)\mathcal F^+(s)+s\frac{1-r}{1-s}(1-F(0))\right]}.
\end{equation}
\end{corollary}

\begin{proof}
It follows from \eqref{E:Lx} that 
\[
(1-\zeta)L\{x\}=P(X_1=x)+r P(X_1>x).
\]
Thus, if $\mathcal L(s)=\mathbb E[s^{H_1},\sup_{n\ge0}S_n>0]$, then
\begin{align}\label{E:scriptL}
(1-\zeta)\mathcal L(s)&=\sum_{x=1}^\infty s^x P(X_1=x)+(1-r)\sum_{x=1}^\infty s^x\sum_{m=x+1}^\infty P(X_1=m)\notag\\
&=\mathcal F^+(s)+(1-r)\sum_{m=1}^\infty P(X_1=m)\left[\frac{s}{1-s}-\frac{s^m}{1-s}\right]\notag\\
&=\left(1-\frac{1-r}{1-s}\right)\mathcal F^+(s)+s\frac{1-r}{1-s}(1-F(0)).
\end{align} 
From \eqref{E:lawofsup} and \eqref{E:defpsi} we see that 
\begin{equation}
\mathcal M(s)=(1-p)\sum_{n=0}^\infty\mathcal L^n(s).
\end{equation}
Using \eqref{E:scriptL} and \eqref{E:findp}, the result follows.
 
\end{proof}

\section{An example from queuing theory}\label{S:example}

We now apply the above result to the abstract queue derived from the tandem queue described in subsection \ref{SS:furthermot}. That is, $N_1,N_2,\dots,$ and $M_1,M_2,\dots,$ are independent sequences of i.i.d. random variables, where $N_k$ counts the number of customers served by Server 1 during the $k$-th busy period, and $M_k$ is a geometric random variable with distribution given by \eqref{E:geo}. Since Server 1's queue is M/M/1, with load $a\equiv \beta/\alpha$ assumed to be less than 1, the distribution of $N_1$ is well known. We have, \cite[II.2]{Cohen:1982},  

\begin{equation}\label{E:distn}
P\{N_1=k\}=\frac{1}{2k-1}\binom{2k-1}{k}\frac{a^{k-1}}{(1+a)^{2k-1}},\quad k=1,2,\dots.
\end{equation}
These probabilities have the generating function
\begin{equation}
U(s)=\sum_{k=1}^{\infty}P\{N_1=k\}s^k
=\frac{1+a}{2a}\left(1-\sqrt{1-\frac{4as}{(1+a)^2}}\right), 
\end{equation}
which converges for $|s|<1+(1-a)^2/4a$. It follows readily that 
\begin{equation}\label{E:meanN}
\mathbb E[N]=\frac{1}{1-a}.
\end{equation}

The associated sequence of waiting times $\{Z_k\}$ is given by $Z_0=0$ and 
\[
Z_{k+1}=\max(0,Z_k+N_{k+1}-M_{k+1}), \quad k=0,1,\dots.
\]

Since $\mathbb E[N_k]=\alpha/(\alpha-\beta)$ (cf. \eqref{E:meanN}), and $\mathbb E[M_k]=\alpha/\gamma$, it follows that the G/G/1-queue has load
\[
b=\frac{\gamma}{\alpha-\beta}.
\]  
We have $b<1$ if and only if $\beta+\gamma<\alpha$, i.e., in the language of microtubules, if the average time for hydrolysis and dissociation of a monomer is less than the average time between the arrival of ATP(GTP)-bound monomers.

We now apply our random walk results from above. Let 
\[
X_k=N_k-M_k,\quad
\text{ and }\quad r=\frac{\alpha}{\alpha+\gamma}.
\] Then $b<1$ implies
\begin{equation}\label{E:load}
 \frac{1}{1-a}-\frac{r}{1-r}<0.
\end{equation}
Furthermore, 
\[
P(M_1\ge k)=r^k.
\]
Thus, if $V$ denotes the generating function of $M_1$, then
\begin{equation}\label{E:genT}
V(s)=\mathbb E[s^{M_1}]=\frac{1-r}{1-rs},\quad |s|<1/r.
\end{equation}
Denote $F$ the common distribution of $X_k$. Then $F(x)=\xi r^{-x}$, $x=0,-1,\dots,$ with $\xi=U(r)$,
\[
\mathbb E[X_1]=\mathbb E[N_1]-\mathbb E[M_1]=\frac{1}{1-a}-\frac{r}{1-r}<0,
\]
and 
\[
r+(1-r)\mathbb E[X_1]=\frac{1-r}{1-a}.
\] 
Moreover,
\[
\mathcal F^+(s)=V(1/s)[U(s)-U(r)],
\]
and 
\[
\mathcal M(s)=\frac{1-\zeta-\frac{1-r}{1-a}}{1-\zeta-(1-r)s\frac{1-U(s)}{1-s}}.
\]
A series representation for $\zeta$ is readily found using Lemma 2 in \cite[XVIII.3]{Feller:1971} and the fact that the constant term in $[U(s)V(1/s)]^n$ equals $P(X_1+\cdots+X_n=0)$.

\end{document}